\documentclass[11pt]{article}
\usepackage{maa-monthly}
\usepackage[all]{xy}
\usepackage{amssymb}
\usepackage{amsthm}
\usepackage{amsmath}
\usepackage{amscd,enumitem}
\usepackage{verbatim}
\usepackage{eurosym}
\usepackage{graphicx}
\usepackage{dsfont}
\usepackage{stmaryrd}
\usepackage{color}
\usepackage{float}
\usepackage{longtable}
\usepackage{dcolumn}
\usepackage[mathscr]{eucal}
\usepackage[all]{xy}
\usepackage{hyperref}
\usepackage[title]{appendix}
\usepackage[usenames,dvipsnames]{xcolor}
\usepackage{bbm}
\usepackage{multirow}
\usepackage{caption}
\usepackage[justification=centering]{caption}
\newtheorem{theorem}{Theorem}
\newtheorem{corollary}[theorem]{Corollary}
\newtheorem{lemma}[theorem]{Lemma}

\newtheorem*{questions}{Problems}

\theoremstyle{definition}

\theoremstyle{remark}
\newtheorem*{remark}{Remark}
\newtheorem*{example}{Example}

\newcommand{\J}{\mathcal{J}}
\newcommand{\E}{\mathcal{E}}
\newcommand{\Z}{\mathbb{Z}}
\newcommand{\Q}{\mathbb{Q}}

\newcommand{\ch}{\mathrm{char}}
\newcommand{\F}{\mathbb{F}}

\newcommand{\AGM}{\mathrm{AGM}}
\newcommand{\class}{\mathrm{class}}

\newcommand{\R}{\mathbb{R}}

\newcommand{\SL}{\operatorname{SL}}

\newcommand{\C}{\mathbb{C}}

\newcommand{\leg}[2]{\genfrac{(}{)}{}{}{#1}{#2}}

\catcode`,\active

\catcode`\,12

\numberwithin{equation}{section}

\begin{document}

\title{AGM and Jellyfish Swarms of Elliptic Curves}
\markright{AGM and Jellyfish Swarms of Elliptic Curves}
\author{Michael J. Griffin, Ken Ono,  Neelam Saikia, and Wei-Lun Tsai}

\maketitle

\begin{abstract} The classical $\AGM$ produces wonderful infinite sequences of arithmetic and geometric means with common limit. For finite fields $\F_q,$ with $q\equiv 3\pmod 4,$ we introduce a finite field analogue $\AGM_{\F_q}$ that
spawns directed finite graphs instead of infinite sequences. The compilation of these graphs reminds one of a {\it jellyfish swarm}, as the
3D renderings of the connected components resemble {\it jellyfish} (i.e., tentacles connected to a bell head).
These swarms turn out to be more than the stuff of child's play; they are taxonomical devices in number theory. Each  jellyfish is an isogeny graph of elliptic curves with isomorphic 
 groups of $\F_q$-points, which can be used to prove that each swarm has at least
 $(1/2-\varepsilon)\sqrt{q}$ jellyfish.
 This interpretation also gives a description of the  {\it class numbers} of Gauss, Hurwitz, and Kronecker which is akin to counting types of spots on jellyfish.
\end{abstract}

\section{Arithmetic and geometric means.}

Beginning with positive real numbers $a_1:=a$ and $b_1:=b,$ the $\AGM_{\R}$ inductively produces a sequence of pairs
$\AGM_{\R}(a,b):=\{(a_1, b_1), (a_2, b_2),\dots\},$ consisting of 
arithmetic and geometric means. Namely,  for $n\geq 2,$ we let
$$
a_n:=\frac{a_{n-1}+b_{n-1}}{2} \ \ \ \ \ {\text {\rm and}}\ \ \ \ \ b_n:=\sqrt{a_{n-1}b_{n-1}}.
$$
For $n \geq 2,$ we have the elementary inequality $a_n\geq b_n.$ At a deeper level, the classical theory of the $\AGM_\R$ (for example, see Chapter 1 of \cite{Borwein}) establishes that these rapidly converging sequences have a common limit 
$ \lim_{n\rightarrow +\infty} a_n=\lim_{n\rightarrow +\infty}b_n.$

In 1748, Euler \cite{Borwein} employed 
$\AGM_{\R}(\sqrt{2},1)$
as a remarkable device for rapidly computing digits of $\pi.$
Namely, he showed that  $\pi = \displaystyle{\lim_{n\rightarrow +\infty} p_n},$ where
$$p_n:= \frac{a_n^2}{1-\sum_{i=1}^n 2^{i-2}(a_i^2-b_i^2)}.$$
Although the first three terms
$p_1=4, \ p_2=3.18767\dots,\ \ {\text {\rm and}} \ \ \ p_3=3.14168\dots$
are quite satisfying, the next two terms
$$
 p_4=3.14159265389\dots \ \ \ {\text {\rm and}}\ \ \
 p_5=3.14159265358979323846\dots
 $$
 are even more astounding as they give 11 and 20 decimal places of $\pi$ respectively.

Is there a finite field analogue of $\AGM_{\R}?$ If so, what number theoretic secrets does it reveal?
 The non-existence of many square-roots in $\F_q$ poses an obvious obstacle.  One solution would be to consider a process where the finite fields grow in size, 
allowing for the existence of square-roots. However, a second issue arises. Namely, what defines the correct choice of square-root? 
Over $\R,$ the convention of taking positive square-roots guarantees that $\AGM_{\R}$ never ventures beyond $\R.$ Consequently,
over finite fields we seek situations where there are unique choices of square-root that similarly avoid the need for field extensions. 

These requirements hold for
finite fields $\F_q$, where
$q=p^m\equiv 3\pmod 4$ with $p$ prime. These fields enjoy the property that $-1$ is not a square, which corresponds to the fact that $i=\sqrt{-1}$ is not a real number.
For such a field $\F_q,$ we let $\phi_q(\cdot)$ be its quadratic residue symbol (the usual Legendre symbol $\leg{\cdot}{p}$ when $q=p$ is prime).
We then define $\AGM_{\F_q}(a,b)$ for pairs
$a, b\in \F_q^{\times}:=\F_q\setminus \{0\},$ with $a\neq \pm b$ and  $\phi_q(ab)=1.$  This input data gives $a_1:=a$ and $b_1:=b,$ and for $n\geq 2$ we let
\begin{equation}
a_n:=\frac{a_{n-1}+b_{n-1}}{2} \ \ \ \ \
{\text {\rm and}}\ \ \ \ \
b_n:=\sqrt{a_{n-1}\cdot b_{n-1}},
\end{equation}
where $b_n$ is the unique square-root with $\phi_q(a_nb_n)=1.$ 
Although $a_{n-1}b_{n-1}$  has two square-roots,  only one  choice satisfies  $\phi_q(a_n b_n)=1$ as $\phi_q(-1)=-1$.  Therefore, we obtain a sequence of  pairs
$$
\AGM_{\F_q}(a,b):=\{(a_1,b_1), (a_2, b_2),\dots \}.
$$

Let's consider the case of $\F_7.$  Half of the 12 pairs that appear in some $\AGM_{\F_7}(a,b)$ form a single
$\AGM$-orbit
$$
\AGM_{\F_7}(1,2)=\{\overline{(1, 2), (5, 3), (4, 1), (6, 5), (2, 4), (3, 6)},\dots\}
$$
(Note. The overlined pairs form a repeating orbit.).
The other 6 pairs  lead to this orbit after a single step. For example, we have
\begin{equation}\label{exampleF7}
\AGM_{\F_7}(6,3 )=\{(6, 3), \overline{(1, 2), (5, 3), (4, 1), (6, 5), (2, 4), (3, 6)},\dots\}.
\end{equation}
The compilation $\J_{\F_7}$ of all such sequences forms a connected directed graph.

 \begin{center}
\includegraphics[height=35mm]{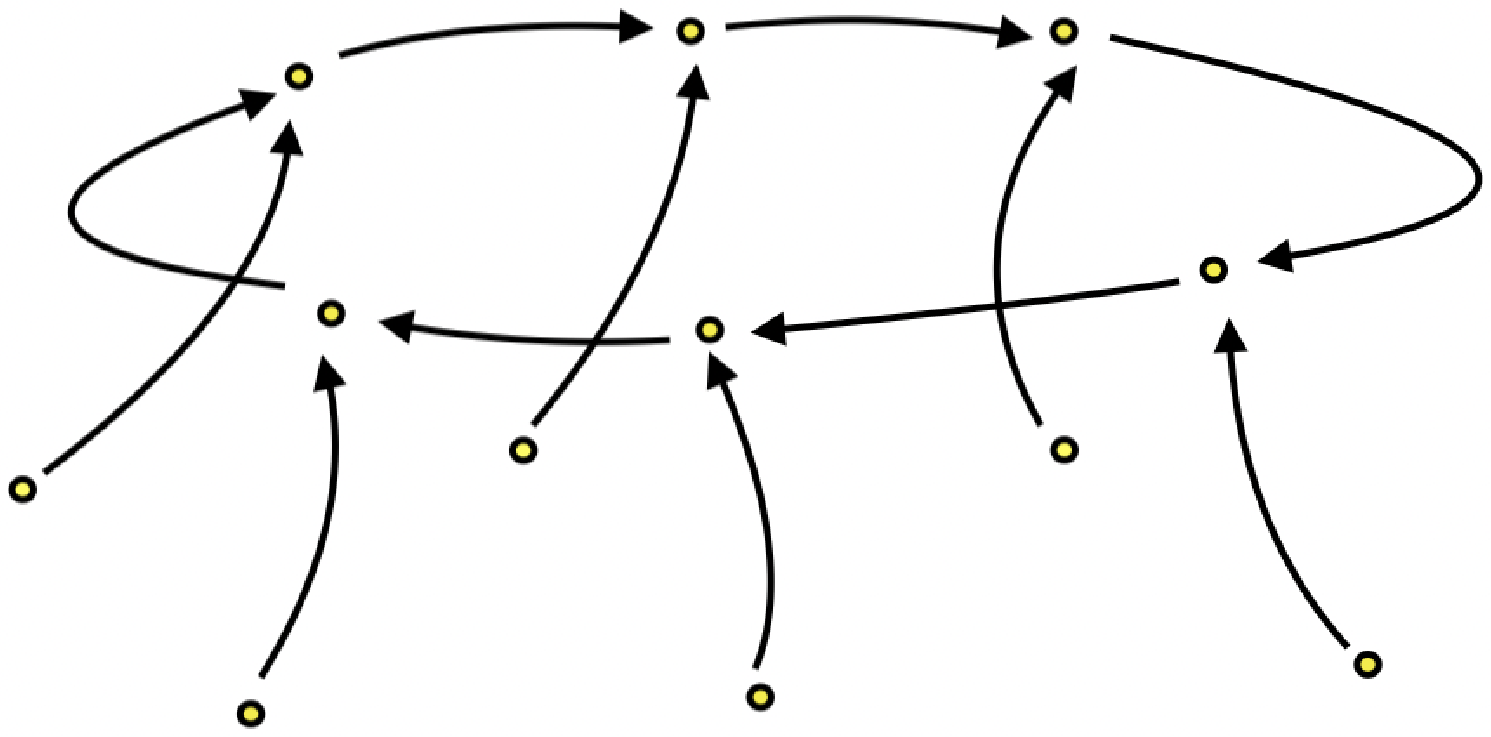}
\captionof{figure}{3D rendering of $\J_{\F_7}$}\label{figure1}
\end{center}

This example is typical.
The compilation of the $\AGM_{\F_q}$ sequences is always a disjoint union of connected directed graphs. The nodes
are {\it admissible ordered pairs}\footnote{There are no loops (i.e., $(a,b)\mapsto (a,b)$) as nodes of the form $(a,a)$ are not allowed.} $(a,b),$ where $a, b \in \F_q^{\times},$
with $a\neq \pm b,$ and $\phi_q(ab)=1.$  Moreover, $(a,b)\mapsto (c,d)$ is an edge if and only if
$c=(a+b)/2$ and
$d=\sqrt{ab},$
with $\phi_q(cd)=1.$  These connected components are like {\it jellyfish}, as their 3D renderings turn out to be unit length tentacles leading to a {\it bell head} cycle.
Hence, we playfully refer to the compilation of the $\AGM_{\F_q}$ sequences
\begin{equation}
\J_{\F_q}:= J_{1}\sqcup J_2 \sqcup \dots \sqcup J_{d(\F_q)}
\end{equation}
as the {\it jellyfish swarm} for $\F_q,$ where the $J_1, J_2,\dots, J_{d(\F_q)}$ are the individual jellyfish which make up the swarm.

Let's summarize some basic facts about $\AGM_{\F_q}$ and the jellyfish swarms $\J_{\F_q}.$ 

\begin{theorem}\label{EZ}
If $\F_q$ is a finite field with $q\equiv 3\pmod 4,$ then the following are true.

\smallskip
\noindent
(1) The $\AGM_{\F_q}$ algorithm is well-defined.

\smallskip
\noindent
(2) The jellyfish swarm $\J_{\F_q}$ has $(q-1)(q-3)/2$ nodes.

\smallskip
\noindent
(3) Every jellyfish has a bell head with length one tentacles pointing to 
each node.

\smallskip
\noindent
(4) If $N_n(\F_q)$ denotes the number of jellyfish with $n$ nodes, then $(q-1) \mid n N_n(\F_q).$
\end{theorem}

\begin{proof}[Proof of Theorem~\ref{EZ}] \ \ 
\smallskip

\noindent 
(1) If $(a,b)$ is admissible, then we must show that the next pair $(c,d)$ generated by $\AGM_{\F_q}$ is  also admissible.  It is clear that $cd\neq 0$ since $ab\neq 0$ and $a\neq \pm b.$ If $c=\pm d,$ then 
$(a-b)^2=a^2-2ab+b^2=4c^2-4d^2=0,$
 which in turn implies the contradiction that $a=b.$  Therefore, $\AGM_{\F_q}$ is well-defined. 

\smallskip
\noindent
(2) We compute the number of admissible pairs $(a, b)$. There are $q-1$ choices for $a$,
and $(q-3)/2$ choices of $b$ with $a\neq \pm b$ which additionally satisfy the quadratic residue condition $\phi_q(ab)=1.$

\smallskip
\noindent
(3)  Each $\AGM_{\F_q}$ sequence eventually enters a repeating cycle (i.e., the bell head). Suppose that  $(a,b)$ is in this orbit. Reversing $\AGM_{\F_q}$ to find its parents, say $(A,B),$
we have $A+B=2a$ and $AB=b^2.$ Therefore, $(A-B)^2=4a^2-4b^2, $ and so  $\phi_q(a^2-b^2)=1$ with two square-roots.  If $(A,B)$ is the parent in the cycle, then the other solution is $(B,A),$ and is not in the cycle. Therefore, $(a, b)$ has exactly one attached tentacle. To see that this tentacle has length $1$, we use the assumption that $(A,B)$ is in the cycle, and so has a parent of its own. Repeating the argument above, we have $\phi_q(A^2-B^2)=1,$ which in turn gives $\phi_q(B^2-A^2)=-1.$ This means that $(B,A)$ does not have a parent. 

\smallskip
\noindent
(4)  Each $\alpha\in \F_q^{\times}$ induces an automorphism on $\J_{\F_q}$ defined by
$(a,b)\mapsto (\alpha a, \alpha b),$ as $\alpha\cdot \AGM_{\F_q}(a,b)=\AGM_{\F_q}(\alpha a, \alpha b).$ As there are no fixed admissible pairs provided that $\alpha\neq 1,$
we find that the orbit of a node under these automorphisms has size $q-1$.
As these automorphisms permute the jellyfish with fixed size, the claim follows. 
\end{proof}

Theorem~\ref{EZ} inspires many natural questions.  
For example, how small (resp. large) are the jellyfish in a general swarm?
This appears to be a very difficult question.
For instance, there are $\AGM_{\F_q}$-orbits that are much shorter than $q,$ such as the length 9 
$$\AGM_{\F_{67}}(1,17)=\{\overline{(1,17), (9, 33)\dots, (65, 15), (40,29)},\dots\},
$$
as well as those that are much longer than $q,$ such as the length 410
$$\AGM_{\F_{83}}(1,3)=
\{\overline{(1,3), (2, 13),\dots, (37,12), (66,19)},\dots\}.
$$
 These examples correspond to a tiny 18 node jellyfish in $\J_{\F_{67}}$, and
 a gigantic 820 node jellyfish in $\J_{\F_{83}}.$
As another question, what can be said about
$d(\F_q),$ the number of jellyfish in $\J_{\F_q}$? 
Table~\ref{table:1} illustrates the oscillatory behavior of $d(\F_q).$  This erratic sequence
does not appear to settle into a predictable pattern as $q\rightarrow +\infty.$
Indeed, there are many astonishing examples of disproportionate consecutive values, such as $d(\F_{479})= 18$ and
$d(\F_{487})=359.$
 The only clear observation is that the $d(\F_{p^m})$ grow rapidly with $m$ when $p$ is fixed. For instance, we have
\begin{displaymath}
\begin{split}
d(\F_3)=0 \ \ &\rightarrow\ \ d(\F_{3^3})=39 \ \  \rightarrow \ \ d(\F_{3^5})=1210,\\
d(\F_7)=1 \ \ &\rightarrow \ \ d(\F_{7^3})=1539\ \ \rightarrow \ \ d(\F_{7^5})=876713,\\
d(\F_{11})=3 \ \ &\rightarrow \ \ d(\F_{{11}^3})=8778\ \ \rightarrow d(\F_{11^5})=25558635.\\
\end{split}
\end{displaymath}
We shall see that the theory of elliptic curves offers deep insight into these questions.

\begin{table}[h]
\centering
\begin{tabular}{|r|cc|cc|cc|cc|cc|cc|cc|cc|}
\hline \rule[-3mm]{0mm}{6mm}
$q$    && 3 && 7 && 11  && 19 && 23 && 31 && 43 && 47  \\   \hline
$d(\F_q)$ && 0 && 1 && 3 && 8 && 5 && 10 && 7 && 4 \\
\hline  
\hline \rule[-3mm]{0mm}{6mm}
$q$    && 59 && 67 && 71 && 79 && 83 && 103 && 107 && 127 \\   \hline
$d(\F_q)$ && 7 && 30 && 25 && 18 && 6 && 41 && 9 && 54  \\
\hline
\hline \rule[-3mm]{0mm}{6mm}
$q$     && 131 &&  139 && 151 && 163 && 167 && 179 && 191 && 199 \\   \hline
$d(\F_q)$ && 46 && 33 && 45 && 38 && 11 && 14 && 14 && 101\\
\hline
\hline \rule[-3mm]{0mm}{6mm}
$q$     && 211 && 223  && 227 && 239 && 251 && 263 && 271 && 283  \\   \hline
$d(\F_q)$  && 120 && 18 && 12 && 40 && 31&& 17 && 34 && 35  \\
\hline
\end{tabular} 
\begin{center}\caption{\label{table:1} $d(\F_q)$ for pimes $q$}\end{center}
\end{table}

\vspace{-0.35in}

\section{Jellyfish swarms organize elliptic curves.}

Computing arithmetic and geometric means over $\F_q$ might seem like mere child's play.
However, it turns out that this arithmetic process is a taxonomical device in number theory which
organizes elliptic curves.

An {\it elliptic curve} $E$ over a field $\F$ can be thought of as a cubic equation of the form 
$$y^2=f(x)=x^3+ax^2+bx+c,$$ where $a, b, c\in \F,$ and $f$ has non-zero discriminant.
If $E(\F)$ denotes the $\F$-rational points of $E$, including the identity ``point at infinity'' $O,$ then
 $E(\F)$ naturally forms an abelian group via the well-known ``chord-tangent law.'' 
The group law can be described by asserting that three colinear points on an elliptic curve sum to the identity $O$. Number theorists are deeply interested in these groups of rational points.

If $F$ is a number field (i.e., a field which has finite degree over $\Q$), then
a classical theorem by Mordell and Weil asserts that $E(\F),$ the Mordell-Weil group of $E/\F$, is finitely generated.
The special case where $\F=\Q$ is the subject of two frequently cited {\it Monthly} articles.
The beautiful 1993 article by
Silverman \cite{silverman2} on the representation of positive integers as sums of two rational cubes 
describes the intimate relationship between Ramanujan's taxi-cab numbers and positive rank elliptic curves (i.e., curves with infinitely many $\Q$-rational points).
The famous 1991 article by Mazur \cite{mazur} promotes the conjectured ``Modularity'' of elliptic curves over $\Q$ (formerly known as the Taniyama-Weil Conjecture).  This conjecture is now known to be true, largely thanks to
 the work of Wiles and Taylor \cite{taylorwiles, wiles}, which was a celebrated ingredient in the proof of Fermat's Last Theorem.

For the case of finite fields, it turns out
that the jellyfish swarms $\J_{\F_{q}}$ organize elliptic curves. One can think of the nodes as spots on the jellyfish, and these spots will be mapped to curves.
These swarms are coverings of networks of special Legendre elliptic curves
 when $q\equiv 3\pmod 4$ and $p\geq 7.$
 To be precise, for $\lambda \in \F_q\setminus \{0, 1\},$ we recall the {\it Legendre normal form} elliptic curve
 \begin{equation}
 E_{\lambda}: \ \ y^2 = x(x-1)(x-\lambda).
 \end{equation}
 Isomorphism classes of elliptic curves are distinguished by their {\it $j$-invariants}, and for $E_\lambda$ we have
  \begin{equation}\label{jinvariant}
  j(E_\lambda)=2^8\cdot\frac{(\lambda^2-\lambda+1)^3}{\lambda^2(\lambda-1)^2}.
  \end{equation}
 For an introduction to elliptic curves over finite fields, the reader may consult Chapter 4 of
 \cite{Washington}.
 
The jellyfish swarm $\J_{\F_q}$ organizes elliptic curves via the map $\Psi_{\F_q}: \J_{\F_q} \mapsto \E_{\F_q},$ where $\E_{\F_q}$ is the set of Legendre curves over $\F_q,$ and
 \begin{equation}
\Psi_{\F_q}(a,b):=E_{\lambda(a,b)},
\end{equation}
where $\lambda(a,b):=b^2/a^2.$ 
For instance,  (\ref{exampleF7}) gives
\begin{displaymath}
\begin{split}
\Psi_{\F_7}(\AGM_{\F_7}(6, 3))&=\Psi_{\F_7}\left(\{(6, 3), \overline{(1, 2), (5, 3), (4, 1), (6, 5), (2, 4), (3, 6)},\dots\}\right)\\
&=\left\{E_2, \overline{E_4, E_4, E_4, E_4, E_4, E_4},\dots\right\}=\left \{E_2, \overline{E_4},\dots \right\}.
\end{split}
\end{displaymath}

As the values $\lambda(a,b)=b^2/a^2$ cover the squares in $\F_q \setminus \{0,1\},$ it is natural ask what special features are shared by curves of the form $E_{\lambda^2}.$ It turns out that these curves are distinguished by the 2-Sylow subgroups
of their $\F_q$-rational points. 

\begin{lemma}\label{Sylow2}  Suppose that $\F_q$ is a finite field with $q\equiv 3\pmod 4.$ If $\lambda\in \F_q\setminus\{0,1\},$ then the 2-Sylow subgroup of
$E_{\lambda^2}(\F_q)$  is of the form
 $\Z/2\Z \times \Z/2^{2+b}\Z,$ where $b\geq 0.$
 \end{lemma}
 \begin{proof}[Proof of Lemma~\ref{Sylow2}] Elliptic curves with four 2-torsion points (i.e., including the identity) can be written in the form
 $$
 E: \ \ y^2=(x-\alpha)(x-\beta)(x-\gamma).
 $$
 The non-trivial 2-torsion points  correspond to the roots of the cubic. They are the points $(\alpha, 0), (\beta, 0)$ and $(\gamma, 0).$
 For such curves, the classical {\it 2-descent lemma} (see p. 47-49 of \cite{Koblitz} or  p. 315 of \cite{silverman})
 says that a non-zero point $P=(x_0,y_0)\in E(\F_q)$ satisfies $P=2Q,$ where
 $Q\in E(\F_q),$ if and only if 
 $$x_0-\alpha,\ \ x_0-\beta, \ \ x_0-\gamma \in \F_q^2.
 $$
 Here we have $\alpha=0, \beta=1,$ and $\gamma=\lambda^2$. 
Since $q\equiv 3\pmod 4,$ then exactly one of $(1,0)$ and $(\lambda^2,0)$ is in
$2E_{\lambda^2}(\F_q)$, as exactly one of $\pm(1-\lambda^2)$ is a square. On the other hand, $(0,0)\not \in 2E_{\lambda^2}(\F_q)$ by the 2-descent lemma because $-1$ is not a square.
Therefore, the $\Z/4\Z$ rank of $E(\F_q)$ is 1, which means that $E(\F_q)$ contains $\Z/4\Z$ but not $\Z/4\Z \times \Z/4\Z.$
\end{proof}

The $\J_{\F_q}$ organize the curves in $\Psi_{\F_q}(\J_{\F_q})$ as unions of explicit {\it isogeny graphs}.
An \textit{isogeny} between two elliptic curves is a special map $\Phi$ (called a \textit{morphism}) that preserves the identity element, is given  by rational functions $\Phi=(u(x,y),v(x,y)),$ and is a homomorphism on $\F_q$-points with finite kernel. 
The isogeny graph structure of $\Psi_{\F_{q}}(\J_{\F_q})$ is provided by the following theorem.

\begin{theorem}\label{main}
If $\F_q$ is a finite field with $q\equiv 3\pmod 4$ and $\ch(\F_q)=p\geq 7,$
then the following are true.

\smallskip
\noindent
(1) We have that
$$
\Psi_{\F_q}(\J_{\F_q}) = \{ E_{\alpha^2}/\F_q \ : \ \alpha \in \F_q\setminus \{0, \pm 1\}\}.
$$
Moreover, each $E_{\alpha^2}\in \Psi_{\F_q}(\J_{\F_q})$ has $q-1$ preimages.

\smallskip
\noindent
(2) For each $1\leq i\leq d(\F_q),$ we have that $\Psi_{\F_q}(J_i)$ is a connected graph\footnote{The proof of Theorem~\ref{main} shows that the swarms are graphs of 2-isogenies. We point interested readers to Sutherland's expository article
\cite{sutherland} for more on the theory of isogeny graphs.}, where an edge  $(a_n,b_n)\rightarrow (a_{n+1},b_{n+1})\in J_i$ is the isogeny
$\Phi_n: E_{\lambda(a_n,b_n)}\rightarrow E_{\lambda(a_{n+1},b_{n+1})}$ defined by
$$
\Phi_n(x,y):=\left(\frac{(a_n x+b_n)^2}{x(a_n+b_n)^2}, -\frac{a_ny(a_nx-b_n)(a_nx+b_n)}{x^2
(a_n+b_n)^3}\right).
$$
Moreover, we have that $\ker(\Phi_n)=\langle(0,0)\rangle.$
\end{theorem}

\begin{proof}[Proof of Theorem~\ref{main}]
\ \ \ 
\smallskip

\noindent
(1) For an admissible pair $(a,b),$ we have $a,b\in\F_q^{\times},$ and $a\neq\pm b.$ Therefore, $b/a\in\F_q\setminus\{0,\pm1\}$ and 
$$\Psi_{\F_q}(\J_{\F_q})\subset
\{ E_{\alpha^2}/\F_q \ : \ \alpha\in \F_q\setminus \{0, \pm 1\}\}.
$$
 On the other hand, if $\alpha\in \F_q\setminus \{0, \pm 1\},$ then one can choose $(a,b)$ such that $a=\pm 1$ and $b=\pm \alpha$  with $\phi_q(ab)=1,$ giving $\lambda(a,b)=\alpha^2.$  Furthermore, each admissible pair $(a,b),$  produces the $q-1$ further admissible pairs $(ka,kb),$  all mapping to $E_{b^2/a^2}.$ Hence, each $E_{\lambda(a,b)}$ has $q-1$ preimages. 

\smallskip
\noindent
(2)  If $(x,y)\neq(0,0)\in E_{\lambda(a_n,b_n)},$ then a brute force calculation gives
\begin{displaymath}
\begin{split}
\frac{(a_n x+b_n)^2}{x(a_n+b_n)^2}
\left(\frac{(a_n x+b_n)^2}{x(a_n+b_n)^2}-1\right)
&\left(\frac{(a_n x+b_n)^2}{x(a_n+b_n)^2}-\frac{b_{n+1}^2}{a_{n+1}^2}\right)\\
&=\frac{a_n^2y^2(a_nx-b_n)^2(a_nx+b_n)^2}{x^4(a_n+b_n)^6}.
\end{split}
\end{displaymath}
This proves that 
$\Phi_n(x,y)\in E_{\lambda(a_{n+1},b_{n+1})}.$
To verify that the map preserves the identity $O=[0,1,0]$ (the point at infinity in projective space),
 we consider the projectivized form 
$\Phi_n(x,y,z):=[\varphi_1(x,y,z),\varphi_2(x,y,z),\varphi_3(x,y,z)],$
where 
\begin{displaymath}
\begin{split}
&\varphi_1:=(a_n+b_n)\\
& \ \ \ \ \ \ \times (a_n^4xy^2+(a_n^4+a_n^2b_n^2)x^3+2a_n^3b_ny^2z+(2a_n^3b_n+2a_nb_n^3)x^2z-2a_nb_n^3xz^2),\\
&\varphi_2:=2a_n^3b_n^2xyz-a_n^5y^3-(a_n^5+a_n^3b_n^2)x^2y,\\
&\varphi_3:=(a_n+b_n)^3(a_n^2y^2z+(a_n^2+b_n^2)x^2z-b_n^2xz^2).
\end{split}
\end{displaymath}
One sees that $\Phi_n(x,y,z)$ preserves the point at infinity $O.$
Furthermore, as $a_n+b_n=2a_{n+1}\neq 0$, we find by inspection that $\Phi_n$ is an isogeny with
 $\ker(\Phi_n)= \langle(0,0)\rangle.$
\end{proof}

What features are shared by the elliptic curves corresponding to the nodes of a single jellyfish? The next corollary offers the  answer.

\begin{corollary}\label{groupstructure}
For each $1\leq i\leq d(\F_q)$,  the following are true.

\smallskip
\noindent
(1) There is an abelian group $G$ such that for all 
$(a_n,b_n)\in J_i$ we have
$$
E_{\lambda(a_n,b_n)}(\F_q)\cong G.
$$
Moreover, the 2-Sylow subgroup of $G$ is  $\Z/2\Z \times \Z/2^{2+b_q(i)}\Z,$
where $b_q(i)\geq 0.$

\smallskip
\noindent
(2) There is a fixed ``trace of Frobenius'' $a_q(i)$ such that for all
$(a_n,b_n)\in J_i$ we have
$|E_{\lambda(a_n,b_n)}(\F_q)|=q+1-a_q(i).$
\end{corollary}

\begin{proof}[Proof of Corollary \ref{groupstructure}]
\ \ \

\smallskip
\noindent
(1) For adjacent pairs $(a_n,b_n),(a_{n+1},b_{n+1})\in J_i,$ Theorem \ref{main} (2) gives an isogeny $$\Phi_n: E_{\lambda(a_n,b_n)} \mapsto E_{\lambda(a_{n+1},b_{n+1})},$$
with $\ker(\Phi_n)=\langle(0,0)\rangle\cong \Z/2\Z.$ It is well-known 
(for example, see Exercise 5.4 on p. 153 of \cite{silverman}) that isogenous elliptic curves over finite fields have the same number of rational points. In particular, the 2-Sylow subgroups of the  groups of $\F_q$ rational points of these two curves have the same order.
Therefore, since $\Phi_n(E_{\lambda(a_n,b_n)}(\F_q))$ is an index $2$ subgroup of $E_{\lambda(a_{n+1},b_{n+1})}(\F_q),$ then as abstract groups we find that
$E_{\lambda(a_n,b_n)}(\F_q)\cong E_{\lambda(a_{n+1},b_{n+1})}(\F_q).$

\smallskip
\noindent
(2) By Theorem \ref{main} (2), we have that $\Psi_{\F_q}(J_i)$ is a connected isogeny graph. As mentioned above, isogenous curves over finite fields have the same number of rational points. Hence, there is a fixed integer $a_q(i),$ known as the trace of Frobenius,  such that
$|E_{\lambda(a_n,b_n)}(\F_q)|=q+1-a_q(i)$
 for each $(a_n,b_n)\in J_i.$ 
 \end{proof}

\begin{example} By Theorem~\ref{EZ},  $\J_{\F_{19}}$ has 144 nodes.
and it turns out that 
$$
\J_{\F_{19}} =J_1\sqcup J_2\sqcup \dots \sqcup J_6 \sqcup J_7 \sqcup J_8
$$
(i.e., $d(\F_{19})=8$),
where the jellyfish can be ordered so that $J_1, J_2,\dots, J_6$ have bell heads with cycle length 6, and $J_7$ and $J_8$
have bell heads with cycle length 18. By Theorem~\ref{main} (1), the nodes in $\J_{\F_{19}}$ map to the eight Legendre curves
with 18 preimages each.
The 6 smaller jellyfish give the isogeny graph depicted below.

\begin{center}
\includegraphics[height=70mm]{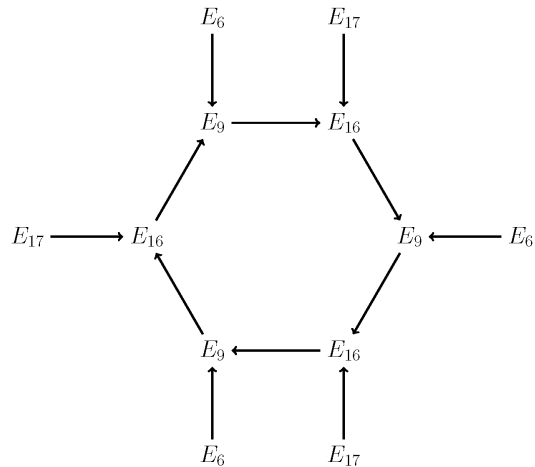}
\captionof{figure}{2D rendering of the isogeny graph of each $J\in  \{J_1,\dots,J_6$\}}
\end{center}

This defines the 3-to-1 covering
$\Psi_{\F_{19}}(J_1) = \dots =
\Psi_{\F_{19}}(J_6)=\{E_6, E_9, E_{16}, E_{17}\}. $
These Legendre curves satisfy
$$E_{6}(\F_{19})\cong E_{9}(\F_{19})\cong  E_{16}(\F_{19})\cong  E_{17}(\F_{19}) \cong \Z/ 2\Z \times \Z/ 12\Z.
$$
We have $a_{19}(1)=\dots=a_{19}(6)=19+1- |\Z/2\Z\times \Z/12\Z|=-4.$  
For $J_7$ and $J_8,$ we obtain the 9-to-1 covering
$\Psi_{\F_{19}}(J_7)= \Psi_{\F_{19}}(J_8)=\{E_4, E_{5}, E_{7}, E_{11}\},$
with 
$$     E_4(\F_{19})   \cong E_{5}(\F_{19})     \cong      E_{7}(\F_{19})  \cong E_{11}(\F_{19}) 
\cong \Z/2\Z \times \Z/8\Z.
$$
Therefore, we have $a_{19}(7)=a_{19}(8)=19+1-|\Z/2\Z\times \Z/8\Z|=4.$

This example shows that individual jellyfish generally include many non-isomorphic curves, as the $j$-invariants (see (\ref{jinvariant})) for the smaller (resp. larger) jellyfish are  $j(E_9)=j(E_{17})=5$ and $j(E_6)=j(E_{16})=15$ (resp. $j(E_7)=j(E_{11})=5$ and $j(E_{4})=j(E_{5})=15$). We shall show that the number of different $j$-invariants, like counting  types of spots on jellyfish, has taxonomic significance.

\smallskip
 \begin{center}
\includegraphics[height=82mm]{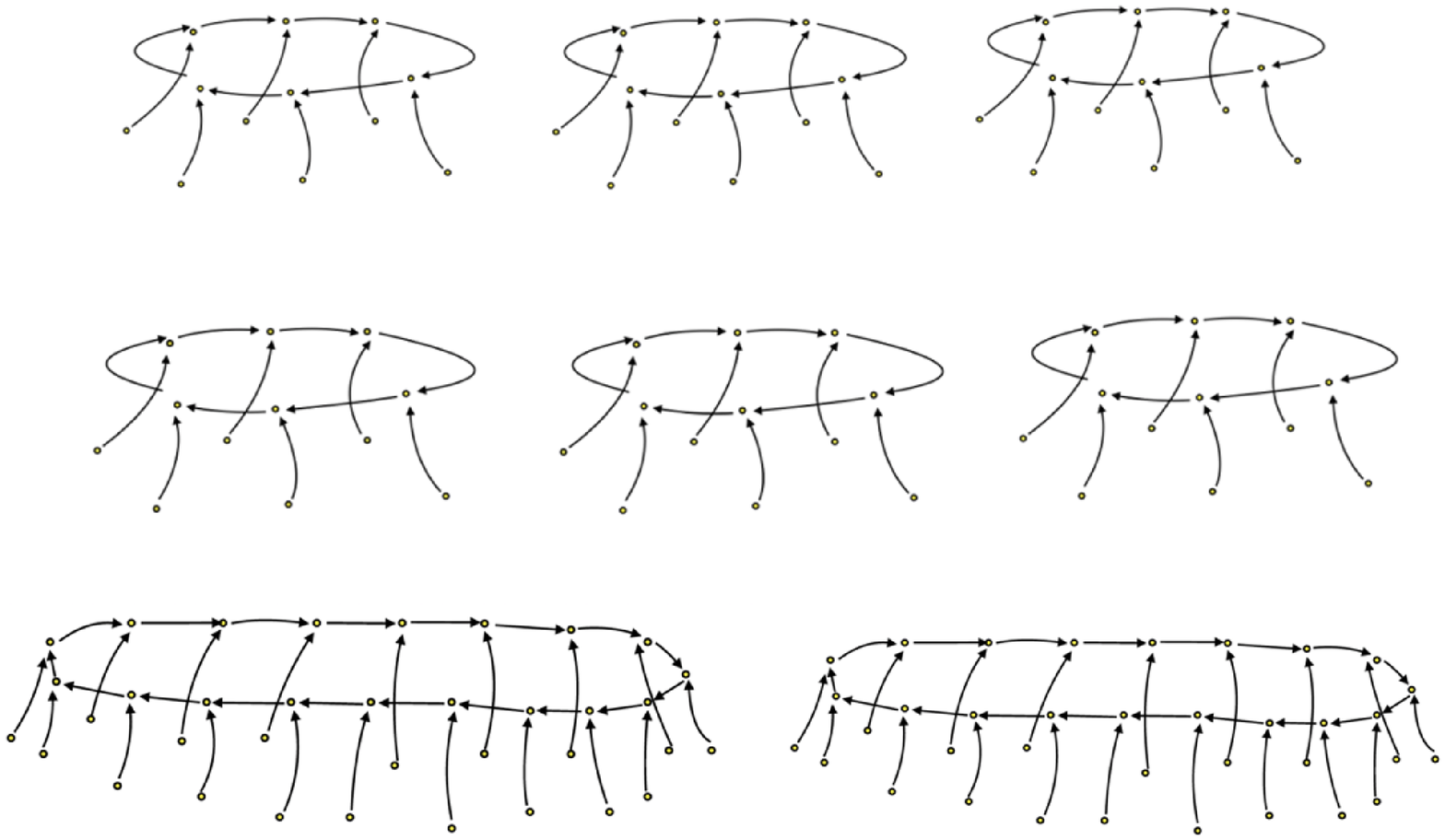} 
\captionof{figure}{$\J_{\F_{19}}$ swarm}
\end{center}
\smallskip
\end{example}

Thanks to the deeper insight offered by Corollary~\ref{groupstructure}, we are able to revisit the baffling numbers $d(\F_q),$
and offer a non-trivial lower bound. 

\begin{theorem}\label{lowerbound}
If $\varepsilon>0$, then for sufficiently large $q\equiv 3\pmod 4$ we have
$$
d(\F_q)\geq \left (\frac{1}{2}-\varepsilon\right)\cdot \sqrt{q}.
$$
\end{theorem}

\begin{remark}
Is this lower bound close to the truth? 
In view of examples such as 
\begin{displaymath}
\begin{split}
&d(\F_{47})=4>\sqrt{47}/2\approx 3.4278,\\
&d(\F_{383})=14>\sqrt{383}/2\approx 9.7851,\\
&d(\F_{983})=25>\sqrt{983}/2\approx 15.6764,\\
&d(\F_{1907})=38>\sqrt{1907}/2\approx 21.8346,\\
&d(\F_{7703})=87>\sqrt{7703}/2\approx 43.8833,
\end{split}
\end{displaymath}
it is tempting to speculate that this lower bound is not much smaller than an optimal bound
which perhaps might be of the form $\gg \sqrt{q}\log \log(q).$
\end{remark}

\begin{proof}[Proof of Theorem~\ref{lowerbound}]
Corollary~\ref{groupstructure} guarantees that
$d(\F_q)$ is at least as large as the number of distinct groups $G$ for which
$E_{\lambda^2}(\F_q)\cong G$
for some $\lambda\in \F_q\setminus \{0, 1\}.$ 
For a group $G$, the proof of Theorem~\ref{ClassNumberRelations} establishes the existence of such a curve provided
$\Z/ 2\Z\times \Z/4\Z\subseteq G$ and there is an $E/\F_q$ for which $E(\F_q)\cong G.$ 

We can construct many such groups.
If  $-2\sqrt{q}\leq s\leq 2\sqrt{q}$ and
$s\equiv q+1\pmod 8,$ then let $m_q(s):=(q+1-s)/2 \equiv 0\pmod 4.$
A classical theorem of R\"uck and Voloch \cite{ruck, voloch} guarantees that one can take $G:=\Z/2\Z \times \Z/ m_q(s)\Z.$  For large $q$, this represents approximately one eighth of the integers in $[-2\sqrt{q}, 2\sqrt{q}].$
Therefore, if $\varepsilon>0,$ then for sufficiently large $q$ we have
\begin{equation}\label{trivialbound}
d(\F_{q}) \geq \left(\frac{1}{2}-\varepsilon\right) \sqrt{q}.
\end{equation}
\end{proof}

\section{Jellyfish swarms and Gauss' Class numbers.} 

The  swarms $\J_{\F_q}$ offer new descriptions
of the {\it class numbers} studied by Gauss, Hurwitz and Kronecker (see \cite{Cox} for more on class numbers).
To make this precise, recall that an \textit{integral binary quadratic form} is a homogeneous degree 2 polynomial 
$$f(x,y):=ax^2+bxy+cy^2\in\Z[x,y].
$$
The \textit{discriminant}\footnote{Discriminants always satisfy $D\equiv 0, 1\pmod 4.$} of $f$ is $D:=b^2-4ac.$ If $a>0$ and $D<0,$ then $f(x,y)$ is called 
{\it positive definite}. Furthemore, $f$ is {\it primitive} if $\gcd(a,b,c)=1.$
For negative discriminants $D$,
the group $\SL_2(\Z)$  acts on
$\mathcal{Q}_D,$ 
 the set of positive definite binary quadratic forms of discriminant $D.$
More precisely, for any $\gamma=\left(\begin{smallmatrix}
u & v\\
r & s
\end{smallmatrix}\right),$ we have
$$
\left(f\circ \gamma\right)(x,y):=f(ux+vy,rx+sy).
$$
Although there are infinitely many primitive binary quadratic forms with discriminant $D,$
Gauss proved that their number of $\SL_2(\Z)$-orbits is finite, and this number is known as \textit{Gauss' class number} $h(D).$ 

Gauss' class numbers lead to the more general  {\it Hurwitz-Kronecker class numbers.}
If $N\equiv 0,3\pmod 4$, then
the Hurwitz-Kronecker class number $H(N)$ is the
class number of positive definite integral binary quadratic forms of discriminant $-N,$
where each class $C$ is counted with multiplicity 
$1/{\text {\rm Aut}}(C).$
If $-N=Df^2,$ where $D$ is a negative fundamental discriminant (i.e., the discriminant of the ring of integers of an imaginary quadratic field),
then $H(N)$ is related to $h(D)$ by  (for example, see p. 273 of \cite{Cox})
$$
  H(N)=\frac{h(D)}{w(D)}\sum_{d|f}\mu(d)
\leg{D}{d}\sigma_1(f/d). 
$$
Here $w(D)$ is half the number of integral units in
$\Q(\sqrt{D})$, and $\sigma_s(n)$ denotes the
sum of the $s^{th}$ powers of the positive divisors of $n$, and $\leg{D}{\cdot}$ is the quadratic Dirichlet character with conductor $D.$

Class numbers have a long and rich history. For example, class numbers play a central role in the study of quadratic forms.
Indeed, if $r_3(n)$ denotes the number of representations of an integer $n$ as a sum of three squares, then Gauss proved that
$$
r_3(n)=\begin{cases} 12H(4n) \ \ \ \ &{\text {\rm if}}\ n\equiv 1,2\pmod 4,\\
           24H(n) \ \ \ \  &{\text {\rm if}}\ n\equiv 3\pmod 8,\\
            r_3(n/4) \ \ \ \ &{\text {\rm if}}\ n\equiv 0\pmod 4,\\
             0 \ \ \ \ &{\text {\rm if}} \ n\equiv 7\pmod 8.
\end{cases} 
$$
Class numbers play even deeper roles in algebraic and analytic number theory, as they are the orders of ideal class groups of rings of integers and orders of imaginary quadratic fields. These groups themselves are the Galois groups of Hilbert class fields. For brevity, we simply say that the study of class numbers continues to drive cutting edge research today.

The jellyfish swarms $\J_{\F_q}$ offer a new interpretation of these class numbers. As the nodes are jellyfish spots, it is quite gratifying to discover that class numbers represent the number of types of spots that appear in a family of jellyfish.
In this analogy, the $j$-invariants distinguish these types of spots.
Namely, for integers $s,$ let
$M_{\F_q}(s)$ be the number of distinct $j$-invariants of curves in the union of
jellyfish with $a_q(i)=s,$ the ``Frobenius trace $s$ family.'' We have the following attractive description which follows from a well-known theorem of Schoof.

\begin{theorem}\label{ClassNumberRelations} Suppose that $\F_q$ is a finite field with $q\equiv 3\pmod 4$ and $p\geq 7.$
If $-2\sqrt{q}\leq s\leq 2\sqrt{q}$ is a non-zero integer with $s\equiv q+1\pmod8,$ then we have
\[
H\left(\frac{4q-s^2}{4}\right)= M_{\F_q}(s).
\]
\end{theorem}

\begin{example}
We revisit the example of $\J_{\F_{19}}$, where
 $J_1,\dots, J_6$ (resp. $J_7$ and $J_8$) are the smaller (resp. larger) jellyfish. We found earlier that the Frobenius trace -4 family is
$$\Psi_{\F_{19}}(J_1) \cup \dots \cup
\Psi_{\F_{19}}(J_6)=\{E_6, E_9, E_{16}, E_{17}\}.
$$
One checks (using (\ref{jinvariant}) that $j(E_9)=j( E_{17})=5$ and $j(E_6)=j(E_{16})=15,$ giving
$M_{\F_{19}}(-4)=2.$ 
We also found that the Frobenius trace $4$ family is
$$ \Psi_{\F_{19}}(J_7)\cup \Psi_{\F_{19}}(J_8)=\{E_4, E_{5}, E_{7}, E_{11}\}.$$
As $j(E_7)=j( E_{11})=5,$ and $j(E_{4})=j(E_{5})=15,$
we also have $M_{\F_{19}}(4)=2.$
Therefore, since $M_{\F_{19}}(\pm 4)=2$,
Theorem~\ref{ClassNumberRelations} gives  $$H\left(\frac{76-(\pm 4)^2}{4}\right)=H(15)=2.$$
\end{example}

\begin{proof}[Proof of Theorem~\ref{ClassNumberRelations}]
Let $E/\F_q$ be an elliptic curve for which $|E(\F_q)|\equiv 0\pmod 8$ and
  $\Z/2\Z\times\Z/2\Z\subseteq E(\F_q).$
We automatically have $\Z/2\Z\times\Z/4\Z\subseteq E(\F_q),$
because these groups can always be described as a direct product of at most 2 cyclic groups.
 Moreover, an application of the 2-descent lemma (for example, see
Proposition 3.3 of \cite{ahlgren-ono}) implies that there is an $\alpha\in\F_q\setminus\left\{0, \pm1\right\}$ for which $E_{\alpha^2}\cong_{\F_q} E.$  Conversely, every $E_{\lambda(a,b)}$ is such an $E$ thanks to Corollary~\ref{groupstructure}.
Therefore, $\J_{\F_q}$ encodes the isomorphism classes of elliptic curves $E/\F_q$ with $\Z/2\Z\times \Z/2\Z\subseteq E(\F_q),$ with the additional property that $q\equiv s+1\pmod 8,$ where
$s:=q+1-\# E(\F_q).$

We consider the $\F_q$ isomorphism classes of such curves with fixed non-zero trace of Frobenius $s$.
If $(a,b), (a',b')\in \J_{\F_q}$ satisfies $j(E_{\lambda(a,b)})=j(E_{\lambda(a',b')}),$ then either $E_{\lambda(a,b)}\cong_{\F_q} E_{\lambda(a',b')},$ or they are non-trivial twists of each other (see Chapter X of \cite{silverman}).
In the latter case, the traces of Frobenius  differ in sign.

Combining these facts, we have that $M_{\F_q}(s)$ is the number of isomorphism classes of elliptic curves
$E/\F_q$ with $\Z/2\Z\times\Z/2\Z\subseteq E(\F_q)$ and trace of Frobenius $s$. 
The rich theory of complex multiplication for elliptic curves is the bridge which connects the counts of such classes with equivalence classes of binary quadratic forms.
Indeed, a well-known deep theorem of Schoof (see Section 4 of \cite{schoof}) asserts that
$H\left(\frac{4q-s^2}{4}\right)$ equals
 the number of such isomorphism classes of elliptic curves. Invoking this theorem completes the proof.

\end{proof}

\section{Analogies between Hypergeometric functions.}\label{Discussion}

Does the classical $\AGM_{\R}$ have more in common with its finite field analogues than
the inductive rules
$$
a_n:=\frac{a_{n-1}+b_{n-1}}{2} \ \ \ \ \ {\text {\rm and}}\ \ \ \ \ b_n:=\sqrt{a_{n-1}b_{n-1}}\ ?
$$
This is indeed the case. It turns out that the results offered above are a byproduct
of remarkable analogies between complex hypergeometric functions and
their finite field analogues. Let us explain.

\subsection{Hypergeometric functions and $\AGM_{\R}.$}

The theory underlying $\AGM_{\R}$ (see Chapter 1 of \cite{Borwein}) is a story involving special integrals and their relationship with Gauss' hypergeometric functions.  To make this precise, for $a>b>0,$ we let
\begin{equation}\label{IR}
I_{\R}(a,b):=\frac{1}{2a}\int_1^{\infty}\frac{dx}
{\sqrt{x(x-1)(x-(1-b^2/a^2))}},
\end{equation}
where the polynomial in the square-root in the denominator of the integrand is tantalizingly\footnote{These models are actually $-1$ quadratic twists of each other.} close to the cubic in the Legendre curve
$$
E_{\lambda(a,b)}: \ \ y^2=x(x-1)(x-b^2/a^2).
$$
It is straightforward to check that
$I_{\R}(a,b)=I_{\R}\left(\frac{a+b}{2}, \sqrt{ab}\right),$ which in turn implies that
$\AGM_{\R}(a,b)=\{(a_1, b_1), (a_2, b_2),\dots\}$ satisfies
\begin{equation}\label{seq-1}
I_{\R}(a_1,b_1)=I_{\R}(a_2,b_2)=\cdots=I_{\R}(a_n,b_n)=\cdots.
\end{equation}

Gauss discovered a beautiful formula for $I_{\R}(a,b)$ in terms of hypergeometric functions.
For $\alpha_1,\alpha_2, \dots \alpha_n$ and $\beta_1, \dots, \beta_{n-1}\in \C,$ these functions are defined by
\begin{equation}\label{classF}
_nF^{\class}_{n-1}\left(\begin{matrix}
\alpha_1, &\alpha_2, &\dots, & \alpha_n\\
~ & \beta_1, &\dots, &\beta_{n-1}
\end{matrix}\mid t\right):=\sum_{k=0}^{\infty}\frac{(\alpha_1)_k(\alpha_2)_k\cdots
(\alpha_n)_k}{(\beta_1)_k\cdots (\beta_{n-1})_k}\frac{t^k}{k!},
\end{equation}
where $(x)_k$ is the Pochhammer symbol defined by
\[{\displaystyle (x)_{k}={\begin{cases}1&\ \ \ \mathrm{if}\ \ \ k=0\\x(x+1)\cdots (x+k-1)&\ \ \ \mathrm{if}\ \ \ k>0.\end{cases}}}
\]
Gauss' theory of elliptic integrals \cite[p. 182]{Husemoller} gives
\begin{equation}\label{relation}
I_{\R}(a,b)=\frac{\pi}{2a}\cdot {_2F^{\class}_1}\left(\begin{matrix}
\frac{1}{2},&\frac{1}{2}\\
~& 1
\end{matrix}\mid 1-\frac{b^2}{a^2}\right),
\end{equation}
which, by letting $a\mapsto (a+b)/2$ and $b\mapsto \sqrt{ab},$ also gives
\begin{equation}\label{relation2}
I_{\R}\left(\frac{a+b}{2},\sqrt{ab}\right)=\frac{\pi}{a+b}\cdot {_2F^{\class}_1}\left(\begin{matrix}
\frac{1}{2},&\frac{1}{2}\\
~& 1
\end{matrix}\mid \frac{(a-b)^2}{(a+b)^2}\right).
\end{equation}
Equating these expressions, we find that $\AGM_{\R}$ leads to the identity 
$$
{_2F^{\class}_1}\left(\begin{matrix}
\frac{1}{2},&\frac{1}{2}\\
~& 1
\end{matrix}\mid 1-\frac{b^2}{a^2}\right)
=\frac{2a}{a+b}\cdot{_2F^{\class}_1}\left(\begin{matrix}
\frac{1}{2},&\frac{1}{2}\\
~& 1
\end{matrix}\mid \frac{(a-b)^2}{(a+b)^2}\right),
$$
which relates $1-b^2/a^2$ with  $\lambda(a+b,a-b)=(a-b)^2/(a+b)^2.$
This identity is a special case (i.e., $\alpha=\beta=1/2$ and $t=(a-b)/(a+b)$) of the  far more general quadratic transformation formula (see \cite[(3.1.11)]{special-function-book}) 
\begin{eqnarray}\label{quadraticR}
{_2F^{\class}_1}\left(\begin{matrix}
\alpha, & \beta\\
~& 2\alpha
\end{matrix}\mid \frac{4t}{(1+t)^2}\right)&\\
=(1+t)^{2\beta}\cdot {_2F^{\class}_1}&\left(\begin{matrix}
\beta+\frac{1}{2}-\alpha,& \beta\\
~& \alpha+\frac{1}{2}
\end{matrix}\mid t^2\right).\notag
\end{eqnarray}

\subsection{Hypergeometric functions and $\AGM_{\F_q}.$}

In view of the previous discussion, we seek a finite field analog of 
$$
I_{\R}(a,b):=\frac{1}{2a}\int_1^{\infty}\frac{dx}
{\sqrt{x(x-1)(x-(1-b^2/a^2))}}.
$$
To this end, one can replace the integral over $\R$ by a sum over $\F_q,$ and
replace the square-root with the quadratic character $\phi_q(\cdot),$
and we can naively declare the finite field analogue to be the sum
\begin{equation}\label{IFq}
I_{\F_q}(a,b):=\sum_{x\in\F_q}\phi_q(x)\phi_q(x-1)\phi_q(x-(1-b^2/a^2)).
\end{equation}
We hope that such sums are values of hypergeometric-type functions.

In his important 1984 Ph.D thesis \cite{GreenePhD}, Greene defined the {\it finite field hypergeometric functions} that do the trick (see \cite{ono, ono-book} for applications). For multiplicative characters\footnote{For multiplicative characters $\chi$, we adopt the convention that $\chi(0):=0.$} $A_1,A_2,\ldots,A_n$ and $B_1,B_2,\ldots,B_{n-1}$ of $\F_q^{\times},$  
he defined
\begin{displaymath}
\begin{split}
{_{n}F_{n-1}}&\left(\begin{array}{cccc}
A_1, & A_2,& \ldots, & A_n\\
~& B_1, &\ldots, &B_{n-1}
\end{array}\mid t\right)_{\F_q}:=\\
&\ \ \ \ \ \ \ \ \ \ \ \ \ \ \ \ \ \ \ \ \ \ \ \ \ \frac{q}{q-1}\sum_{\chi}{A_1\chi\choose\chi}{A_2\chi\choose B_1\chi}\cdots{A_n\chi\choose B_{n-1}\chi}\chi(t),
\end{split}
\end{displaymath}
where the sum is over the multiplicative characters of $\F_q^{\times},$ and ${A\choose B}$ is the normalized Jacobi sum
$$
{A\choose B}:=\frac{B(-1)}{q}J(A,\overline{B}):=\frac{B(-1)}{q}\sum\limits_{t\in\F_q}A(t)\overline{B}(1-t).
$$
This definition was meant to resemble (\ref{classF}), and is based on analogies between Gauss sums and the complex $\Gamma$-function, which interpolates factorials, and the classical Gauss sum expression for Jacobi sums (when $\chi \psi$ is nontrivial)
$$
J(\chi,\Psi)=\frac{G(\chi)G(\psi)}{G(\chi \psi)},
$$
which in turn emulates binomial coefficients.

These functions take an attractive form when $n=2$ (see p. 82 of \cite{Greene}).
If $A, B$ and $C$ are characters of $\F_q$ and $t\in\F_q^\times,$ then
$$
_2F_1\left( \begin{matrix} A, & B\\ \ & C \end{matrix}\ \mid \
t\right)_{\F_q}=\frac{BC(-1)}{q}\cdot \sum_{x\in \F_q} B(x)\cdot \overline{B}C(1-x)\cdot \overline{A}(1-xt).
$$
In particular, if $A=B=\phi_q(\cdot)$ and $\varepsilon_q(\cdot)$ is trivial, then a change of variables gives
\begin{equation}\label{FFq}
 _2F_1(\lambda)_{\F_q}:=\  _2F_1\left ( \begin{matrix}\phi_q, & \phi_q\\ \ & \varepsilon_q\end{matrix}\ | \ \lambda\right)_{\F_q}=-\frac{\phi_q(-1)}{q}\cdot \sum_{x\in \F_q}
\phi_q(x(x-1)(x-\lambda)).
\end{equation}
In analogy with Gauss' integral formulas, which give {\it periods} of elliptic curves,
Greene's functions compute traces of Frobenius over $\F_q.$ Indeed,
if $\ch(\F_q)>3$ and $\lambda\in \F_q\setminus\{0,1\},$ then (\ref{FFq}) gives
\begin{equation}\label{trace2F1}
|E_{\lambda}(\F_q)|=q+1+q\phi_q(-1)\cdot \, _2F_1(\lambda)_{\F_q}.
\end{equation}

In terms of the desired analogy, 
if $q\equiv 3\pmod 4,$  then (\ref{IFq}) and (\ref{FFq}) gives the counterpart of (\ref{relation})
$$
 I_{\F_q}(a,b)= q\cdot \,_2F_1(1-b^2/a^2)_{\F_q}.
$$
To complete the analogy, we require a quadratic transformation law which plays the role
of (\ref{quadraticR}).
We conclude by stating this recent theorem of
Evans and Greene (see Th. 2 of \cite{EvansGreene}), which precisely offers the desired analogous transformation.

\begin{theorem}\label{2F1Transformation}{\text {\rm [Th. 2 of \cite{EvansGreene}]}}
Suppose that $A,$ $A^2\overline{B},$ and $\phi_q A \overline{B}$ are all nontrivial characters of $\F_q^\times.$ If $t\in \F_q\setminus \{-1\},$ then
\begin{displaymath}
\begin{split}
& _2F_1\left( \begin{matrix} A, & B\\ \ & A^2\end{matrix} \ \mid \ \frac{4t}{(1+t)^2}\right)_{\F_q}\\
&
\ \ =\frac{\overline{A}(4) \phi_qB(-1) G(A^2 \overline{B})\cdot G(\phi_q \overline{A}B)}
{G(\phi_q) G(A)}\cdot B^2(1+t) \, _2F_1\left( \begin{matrix}\phi_q \overline{A}B, & B\\
\ & \phi_q A\end{matrix}\  \mid \ t^2\right)_{\F_q}.
\end{split}
\end{displaymath}
\end{theorem}

\section{Epilogue.}

We hope that the reader agrees that the story presented here is a beautiful amalgamation of  facts about elliptic curves over
$\C$ and over finite fields. It is quite marvelous to find that the hypergeometric functions
of Gauss (in the case of $\R$) and of Greene (in the case of $\F_q$) underlie different features in the theory of  elliptic curves that are captured by sequences of arithmetic and geometric means. We hope that this story encourages readers to learn more about the theory of Gauss' class numbers, elliptic curves, and hypergeometry. We highly recommend D. Cox's  book ``Primes of the form $x^2+ny^2$''  \cite{Cox} and ``Pi and the AGM'' \cite{Borwein} by P. Borwein and J. Borwein.

We aim to entice readers with the following tantalizing problems.

\begin{questions} \ \ \

\noindent
(1) What can one prove about the sizes of the jellyfish in $\J_{\F_q}$?
This question is intimately connected to the unproven Cohen-Lenstra heuristics on the expected behavior
of class groups of imaginary quadratic orders.
%%Note:  The number of Jelly fish corresponding
%%to elliptic curves of Frobenius trace t != 0 is just the
%%index of the subgroup of G=cl(Q(sqrt(t^2-4q)) generated by
%%a prime ideal p_2 of norm 2 (if there are any), and one could
%%assume the class of p_2 is uniformly distributed and use
%%Cohen-Lenstra to predict the shape of G as t varies over
%%[-2sqrt(q),2sqrt(q)] (probably one needs to adjust for cases
%%where 2 divides the discriminant of Q(sqrt(t^2-4q)) but this
%%can also be computed on average.
\smallskip

\noindent
(2) Determine an ``optimal'' function $D(q)$ for which 
$$d(\F_q)\geq D(q).$$
In particular, how close to optimal is the lower bound in Theorem~\ref{trivialbound}?
Is the correct lower bound more like $\gg \sqrt{q}\log \log(q)$?
\smallskip

\noindent
(3) It would be very interesting to define variants of $\AGM$ in situations where choices of square-root are not well-defined, such as
the complex field $\C$ and the finite fields $\F_q$ with $q\not \equiv 3\pmod 4.$

\end{questions}

To conclude, we must confess that the $\AGM_{\F_q}$ jellyfish are merely alluring examples of creatures that inhabit the magnificent kingdom formed out of elliptic curves over finite fields. The beautiful $\AGM_{\F_q}$ sequences innocently offer glimpses of the fascinating theory of isogenies for elliptic curves over finite fields, which form networks, dubbed {\it isogeny volcanoes}. The jellyfish are examples that arise from ``2-volcanoes of height 1.'' Isogeny volcanoes
play important roles in computational number theory  and cryptography. They are often employed as a means of accelerating number theoretic algorithms. They have even been used to quickly compute values of Euler's partition function
\cite{PofN}.
 This important theory has its origins in David Kohel's 1996 PhD  thesis \cite{Kohel}. We invite interested readers to read the delightful expository article \cite{sutherland} by Sutherland.

\begin{acknowledgment}{Acknowledgment.}
The authors thank  the referees, Jennifer Balakrishnan, Hasan Saad, and Drew Sutherland for comments and suggestions that improved this article.
The second
  author thanks  the Thomas Jefferson Fund and the NSF
(DMS-2002265 and DMS-2055118) for their generous support, as well as  the Kavli Institute grant NSF PHY-1748958.
The third author is grateful for the support of a Fulbright Nehru Postdoctoral Fellowship.
\end{acknowledgment}

\begin{biog}

\item[Michael J. Griffin] received the Ph.D. degree in Mathematics from Emory University in 2015. He is an Assistant Professor of Mathematics at Brigham University in Provo, Utah, USA. His research interests  are in number theory. \begin{affil}
Department of Mathematics, 275 TMCB,  Brigham Young University, Provo, UT 84602\\
mjgriffin@math.byu.edu
\end{affil}

\item[Ken Ono] received the Ph.D. degree in Mathematics from UCLA in 1993. He is the Thomas Jefferson Professor of Mathematics at the University of Virginia, Charlottesville, Virginia, USA. His research interests are in number theory.
\begin{affil}
Department of Mathematics, University of Virginia, Charlottesville, VA 22904\\
ken.ono691@virginia.edu
\end{affil}

\item[Neelam Saikia] received the Ph.D. degree in Mathematics from the Indian Institute of Technology in Delhi, India in 2016.  She is a Nehru-Fulbright Postdoctoral Fellow at the University of Virginia, Charlottesville, Virginia, USA. Her research interests are in number theory.
\begin{affil}
Department of Mathematics, University of Virginia, Charlottesville, VA 22904\\
nlmsaikia1@gmail.com
\end{affil}

\item[Wei-Lun Tsai] received the Ph.D. degree in Mathematics from the Texas A\&M University in 2020. He is a Research Postdoctoral Fellow at the University of Virginia, Charlottesville, Virginia, USA. His research interests are in number theory.
\begin{affil}
Department of Mathematics, University of Virginia, Charlottesville, VA 22904\\
wt8zj@virginia.edu
\end{affil}

\end{biog}
\vfill\eject

\end{document}